\documentclass[12pt,leqno]{amsart}
\usepackage{epsf, amsmath, amsfonts, amsthm, amssymb}
\newtheorem{thm}{Theorem}[section]
\newtheorem{lem}[thm]{Lemma}
\newtheorem{prop}[thm]{Proposition}
\newtheorem{cor}[thm]{Corollary}

\theoremstyle{definition}

\newtheorem*{rem}{Remark}

\newcommand{\N}{\mathbb N}

\newcommand{\R}{\mathbb R}

\newcommand{\cA}{\mathcal{A}}
\newcommand{\cB}{\mathcal{B}}

\newcommand{\cF}{\mathcal{F}}

\newcommand{\cU}{\mathcal U}
\newcommand{\cV}{\mathcal V}

\newcommand{\cS}{\mathcal S}

\newcommand{\vp}{\varepsilon}

\newcommand{\bx}{\ensuremath{\boldsymbol{x}}}
\newcommand{\by}{\ensuremath{\boldsymbol{y}}}
\newcommand{\bz}{\ensuremath{\boldsymbol{z}}}

\newcommand{\spa}{{\rm span}}

\newcommand{\dist}{{\rm{dist}}}

\def \eps{\varepsilon}
\def\hangbox to #1 #2{\vskip1pt\hangindent #1\noindent \bhox to
#1{#2}$\!\!$}

\newcommand{\keq}{\!=\!}

\newcommand{\kleq}{\!\leq\!}

\newcommand{\kle}{\!<\!}
\newcommand{\kgr}{\!>\!}
\newcommand{\kin}{\!\in\!}
\newcommand{\ksubset}{\!\subset\!}
\newcommand{\ksupset}{\!\supset\!}

\newcommand{\xt}{\tilde x}

\newcommand{\dz}{\ensuremath{\mathrm{Dz}}} % Dentabilitiy  index
\newcommand{\sz}{\ensuremath{\mathrm{Sz}}} % Szlenk index
\newcommand{\ie}{\textit{i.e.,}\ }

\newcommand{\sder}[1]{\ensuremath{^{(#1)}_S}} % S-derivative
\newcommand{\si}{\ensuremath{\mathrm{I}_S}} % S-index

\title{The Szlenk index of $L_p(X)$}
\begin{document}

\allowdisplaybreaks

\begin{abstract}
We find an optimal upper bound on the values of the weak$^*$-dentability
index $\dz(X)$ in terms of the Szlenk index $\sz(X)$ of  a  Banach space $X$ with separable dual.
Namely, if $\;\sz(X)=\omega^{\alpha}$, for some
$\alpha<\omega_1,$ and $p\in(1,\infty)$, then
\[\sz(X)\le
 \dz(X)\le
\sz(L_p(X))\le \begin{cases}  \omega^{\alpha+1} &\text{ if $\alpha$ is a finite ordinal,}\\
                                                            \omega^{\alpha}       &\text{ if $\alpha$ is an infinite ordinal.}\\
                                                            \end{cases}
\]
\end{abstract}

\author{Petr H\'ajek}
\address{Mathematical Institute\\Czech Academy of Science\\\v Zitn\'a 25\\115 67 Praha 1\\
Czech Republic\\
and Department of Mathematics\\Faculty of Electrical Engineering\\
Czech Technical University in Prague\\ Zikova 4, 160 00, Prague}
\email{hajek@math.cas.cz}

\author{Thomas Schlumprecht}
\address{Department of Mathematics\\Texas A\&M University\\College Station, TX 77843\\
and Department of Mathematics\\Faculty of Electrical Engineering\\
Czech Technical University in Prague\\ Zikova 4, 160 00, Prague}
\email{schlump@math.tamu.edu}

%\date{September 2005}
%\keywords{}

\thanks{The first author's research was  supported in part by
GA\v CR P201/11/0345,
Project Barrande 7AMB12FR003,  and RVO: 67985840.
The second author's research is partially supported by NSF
grants DMS0856148 and DMS1160633.}
\maketitle

\section{Introduction}

Let $X$ be a Banach space.
We say that  the dual $X^*$ is {\em weak$^*$-dentable}
 if for every nonempty bounded
subset $M\subset X^*$ and for every $\vp>0$ there are
$u\in X$ and $a\in\R$ such that the slice
$\{x^*\in M:\ \langle x^*,u\rangle>a\}$
is nonempty and has diameter less than $\vp$.
We say that  $X^*$ is {\em weak$^*$-fragmentable}
if for every nonempty bounded
subset $M\subset X^*$ and for every $\vp>0$ there is a
weak$^*$-open set $V\subset X^*$
such that the intersection $M\cap V$ is nonempty and has diameter
less than $\vp$. In \cite{As} Asplund considered  the property of $X$ that every continuous convex function defined
on an open set of $X$ is Fr\'echet differentiable on a dense $G_\delta$ set, and we call such a space an  {\em Asplund space}.
The following  equivalences between the notions are  stated in \cite{FHHMZ} and gather the results from \cite{As,NP,Phe}.

\begin{thm}\label{T:1.1} \cite[Theorem 11.8, p. 486 ]{FHHMZ}

Let $(X,\|\cdot\|)$ be a Banach space.
Then the following assertions are equivalent:

 {\em (i)} $X^*$ is weak$^*$-dentable.

 {\em (ii)} $X^*$ is weak$^*$-fragmentable.

 {\em (iii)} $X$ is an Asplund space.

 {\em (iv)} Every separable subspace of $X$ has a separable dual.
\end{thm}

This fundamental result has many ramifications,
including  for the investigation  of the  {\em Radon-Nikod\'ym Property}  and the
renorming theory of Banach spaces, see e.g. \cite{Bou,FHHMZ,HMVZ}.

Our object of study in this note is the quantitative relationship
between   weak$^*$-dentability and
weak$^*$-fragmentability. Our results are expressed in terms
of the values of derivation indices, which are naturally associated with the
fragmentation properties.

We begin by  defining  the Szlenk derivation and the Szlenk index that
have been first introduced in \cite{Sz}.

Consider
a real Banach space $X$ and a weak$^*$-compact subset $K$  of $X^*$. For
$\eps>0$ we let $\cV_{(K,\eps)}$ be the set of all relatively weak$^*$-open subsets $V$
of $K$ such that the norm diameter of $V$ is less than $\eps$ and put
$s_{\eps}(K)=K\setminus \cup\{V:V\in\cV_{(K,\eps)}\}.$ Then we define inductively
$s_{\eps}^{\alpha}(K)$ for any ordinal $\alpha$ by
$s^{\alpha+1}_{\eps}(K)=s_{\eps}(s_{\eps}^{\alpha}K)$ and
$s^{\alpha}_{\eps}(K)={\displaystyle \cap_{\beta<\alpha}}s_{\eps}^{\beta}K$, if
$\alpha$ is a limit ordinal. We then define $\sz (X,\eps)$ to be the least ordinal $\alpha$ so
that $s_{\eps}^{\alpha}(B_{X^*})=\emptyset,$ if such an ordinal exists. Otherwise
we write $\sz (X,\eps)=\infty.$ The {\it Szlenk index} of $X$ is finally
defined to be  $\sz (X)=\sup_{\eps>0}\sz (X,\eps).$

If $K$ is weak$^*$-compact and convex, we call a
weak$^*$-slice of $K$ any non empty set of the form $S=\{ x^*\in K,\
x^*(x)>t\},$ where $x\in X$ and $t\in \R$. Then we denote for $\eps>0$ by $\cS_{(K,\eps)} $ the set of all weak$^*$-slices of $K$ of norm diameter less than
$\eps$ and put  $d_{\eps}(K)=K\setminus \cup\{S:S\in\cS_{(K,\eps)} \}$. From this
derivation, we arrive similarly to the {\it weak$^*$-dentability
indices} of $X$ that we denote $\dz(X,\eps)$, for $\eps>0$, and $\dz(X)=\sup_{\vp>0}\dz(X,\eps)$. Since $\cS_{(K,\eps)}\subset \cV_{(K,\eps)}$, for  all $\eps>0$, it follows immediately
that $\dz(X,\eps)\ge \sz(X,\eps)$, and  $\dz(X)\ge \sz(X)$. Our  problem consists of finding an estimate going in the opposite direction.

In the language of indices Theorem \ref{T:1.1} implies
that $\sz (X)\ne\infty$ holds if and only if
$\text{Dz}(X)\ne\infty$. Indeed, the respective index is equal to
$\infty$ if and only if the dual $X^*$ contains
a $w^*$-compact and non empty subset without any $w^*$-open and
nonempty subsets (resp. slices) of diameter less than
some $\eps>0$.

It is now clear that a natural quantitative approach to
Theorem \ref{T:1.1} consists of comparing the values of $\sz (X)$ and $\text{Dz}(X)$.
This problem has received a fair amount of attention in the literature.
The first estimates in this direction were purely existential.
 We recall \cite[Lemma 1.6]{Sz}  that if $X^*$ is separable then $\sz (X)<\omega_1$.
In
\cite[Proposition 2.1]{La2} it is shown, using an approach from descriptive set theory due to B.
Bossard (see \cite{Bo1} and \cite{Bo2}), that there is a universal function
$\psi:\omega_1 \to \omega_1$, such that if $X$ is an Asplund space with
$\sz (X)<\omega_1$, then $\text{Dz}(X)\leq \psi(\sz (X))$.
Using geometrical arguments, Raja \cite[Theorem 1.3]{raja} has proved that
one can use $\psi(\alpha)=\omega^\alpha$ as a growth control function
for every ordinal $\alpha$ (\ie without the restriction $\alpha<\omega_1$).
The best value for $\psi(\omega)$,  namely $\psi(\omega)=\omega^2$ was obtained in \cite[Theorem 4.1]{HL} .

Our main result, Theorem \ref{T:1.2},
 gives the optimal form of $\psi$, for all $\alpha<\omega_1$.
In particular it solves the problem for all separable spaces with separable dual.

\begin{thm} \label{T:1.2} Let $X$ be an  Asplund space and $1<p<\infty$.
 If $\;\sz(X)=\omega^{\alpha}$, for some
$\alpha<\omega_1,$ then
\begin{equation}\label{E:1.2.1}
\sz(X)\le \dz(X)\le
\sz(L_p(X))\le \begin{cases}  \omega^{\alpha+1} &\text{ if $\alpha$ is a finite ordinal,}\\
                                                            \omega^{\alpha}       &\text{ if $\alpha$ is an infinite ordinal.}\\
\end{cases}
\end{equation} \end{thm}

It should be noted \cite[Proposition 5.4]{La2} that if $\sz(X)<\omega_1$,  then the Szlenk index of $X$
must be of the form $\sz (X)=
\omega^{\alpha}$, for some ordinal $\alpha$.
This was noted independently and also for several other indices  in \cite[Corollary 3.10]{AJO}.
The same condition holds for the dentability index, \ie if $\dz(X)<\omega_1$ then
$\text{Dz}(X)=
\omega^{\alpha}$, for some ordinal $\alpha$.
So there are no possible intermediate values of indices between $\omega^{\alpha}$
and $\omega^{\alpha+1}$. Our result shows that the dentability
index is either equal to the Szlenk index, or
if $\alpha$ is finite it may happen that it exceeds Szlenk by just one step.
At the end of our note we indicate
examples showing that both possibilities may occur in the case that $\alpha$ is finite.

It should be also noted that both indices $\sz (X), \text{Dz}(X)$
have found many applications in the  geometry and the structure  of Banach spaces,
renorming theory and nonlinear theory. This regards also the quantitative
estimates of their values, and their relationships.
For more details we refer to the survey
paper of Lancien \cite{La3}.

\section{Proof of the main result}\label{S:2}

The proof of the main theorem, which is given at the end of this section,
requires several ingredients. We are going to review
these ingredients first, together with some necessary technical modifications
needed for our proof.
The main new idea, contained in Lemma \ref{L:2.4} and its Corollary \ref{C:2.5}, consists of
a nonlinear technique for transferring certain trees between pairs of
Banach spaces.

Let us denote by $L_p(X)$ the space of all $X$-valued Bochner integrable
functions on $[0,1]$, equipped with the $L_p$-norm.
By a result of Lancien \cite[Lemma 1]{La3} if $p\in(1,\infty)$
then
\begin{equation}\label{E:2.1}
\text{Dz}(X)\le\sz(L_p(X))
\end{equation}
 for any space $X$ having a  separable dual. The proof in \cite{La3}
is done for $p=2$, but it can be easily adjusted to any $p\in(1,\infty)$.

We now  recall  some standard facts about ordinals and the spaces of continuous function on them.
  We denote by $\omega$  the first infinite
ordinal and by  $\omega_1$ the first uncountable ordinal.
 We always consider  sets of
ordinals as topological spaces equipped with the order topology.

The isomorphic classification of the spaces $C([0,\alpha])$, for $\alpha<\omega_1,$
is due to C. Bessaga and  A. Pe\l czy\'nski \cite[Theorem 1]{BP}.
They have shown that $C([0,\omega^{\omega^\alpha}])$, for $\alpha<\omega_1$,
are pairwise non-isomorphic spaces, and for every
$\omega^{\omega^\alpha}\le\beta<\omega^{\omega^{\alpha+1}}$ there
is an isomorphism
 between $C( [0,\beta]$ and $C\big([0,\omega^{\omega^{\alpha}}]\big)$.
Moreover,
every $C(K)$ space for a countable compact $K$ is isomorphic to one of these
spaces. Samuel \cite[Th\'eor\`em, p.91]{Sa} computed the precise values of the Szlenk index and showed that
\begin{equation}\label{E:2.3}
\sz \big(C([0,\omega^{\omega^\alpha}])\big)=\omega^{\alpha+1},\text{ for all $\alpha<\omega_1$},
\end{equation}
which implies that the Szlenk index
determines the isomorphic classes of the separable
$C([0,\omega^{\omega^\alpha}])$ spaces. Other proofs of this result were given
in \cite{AJO,HL}.

One of the main ingredients of our proof is an alternative description
of the Szlenk  index introduced in \cite{AJO}, which is based
on a derivation and its corresponding index defined
for certain trees in the  space $X$.
This approach has been further developed e.g. in \cite{Ca,FOSZ,OSZ}, and we now  recall some notion introduced there.

Let $X$ be a Banach space. We let  $S_X^{<\omega}\keq
\bigcup_{n=0}^\infty S_X^n$, the set of all finite sequences in $X$,
which includes the sequence of length zero denoted by $\emptyset$. For
$x\kin X$ we shall write $x$ instead of $(x)$, \ie we identify $X$
with sequences of length~$1$ in $X$. A \emph{tree on $S_X$ }is a
non-empty
subset $\cA$ of $S_X^{<\omega}$ closed under taking initial segments: if
$(x_1,\dots,x_n)\kin \cA$ and $0\kleq m\kleq n$, then
$(x_1,\dots,x_m)\kin\cA$. There is a natural partial order $\preceq$
on the elements of the tree $\cA$, which gives $a\preceq b$ if and only if
$a$ is an initial segment of $b$.

Given $\bx\keq (x_1,\dots,x_m)$ and $\by\keq(y_1,\dots,y_n)$ in
$X^{<\omega}$, we write $(\bx,\by)$ for the concatenation of $\bx$
and $\by$:
$$
(\bx,\by)=(x_1,\dots,x_m,y_1,\dots,y_n) .
$$
Given $\cA\subset S_X^{<\omega} $ and $\bx\in S^{<\omega}_X$, we let
$$\cA(\bx)=\big\{\by\in S_X^{<\omega}: (\bx,\by)\in \cA\big\}.$$

Let $S$ be a set consisting of  sequences in $S_X$. In our case $S$ will be the set of
normalized weakly null sequences in $X$. For a tree $\cA$ on $X$ \emph{the $S$-derivative
  $\cA_S'$ of $\cA$ }consists of all finite sequences of two kinds:

1. first kind: $\bx\kin
X^{<\omega}$, for which there is a sequence $(y_i)_{i=1}^\infty\kin S$
with $(\bx,y_i)\kin\cA$ for all $i\kin\N$,

2. second kind: initial segements $(x_1,\dots, x_m)$,  $m\le n$,  where $(x_1,\dots,x_n)$
is a sequence of the first kind.

 Note that
$\cA'_S\ksubset\cA$ and that $\cA'_S$ is also a tree unless it is empty

We define
higher order derivatives $\cA\sder{\alpha}$ for ordinals
$\alpha\kle\omega_1$ by recursion as follows.
\begin{align*}
  \cA\sder{0} = \cA,\,
  \cA\sder{\alpha+1} \!= \big( \cA\sder{\alpha}\big)'_S,\text{ for }\alpha\kle\omega_1,\text{ and }
  \cA\sder{\lambda} = \bigcap _{\alpha <\lambda}
  \cA\sder{\alpha} \text{ for limit ordinals  }\lambda\kle\omega_1 .
\end{align*}
It is clear that $\cA\sder{\alpha}\ksupset \cA\sder{\beta}$, whenever
$\alpha\kleq\beta$, and that $\cA\sder{\alpha}$ is a tree
or empty,
 for all
$\alpha$.  An easy induction also shows that
$$\big(\cA(\bx)\big)^{(\alpha)}_S=\big(\cA\big)^{(\alpha)}_S(\bx) \text{  for all $\bx\in S_X^{<\omega}$ and all ordinals $\alpha$}.$$

%%%%%%%%%%%%%%%%%%%%%%%%%%%%%%%%%%%%%%%%%%

Our proof will rely on the use of trees with the next additional heredity property.
We will say that $\cA$ is a \emph{hereditary tree} (H-tree, for short) if for  every sequence $\bx\in\cA$, every subsequence 
 of $\bx$ is also in $\cA$. Note  that in this case
  all elements of the second kind are also of the first kind and  that $\cA'$ consists therefore of all sequences in $\cA$ which are of the first kind. Taking the $S$-derivative of an H-tree therefore amounts to removing all elements which are not of the first kind.
It is clear that the property of being
an H-tree is preserved under taking $S$-derivatives of any ordinal order.

We now define \emph{the $S$-index $\si(\cA)$ of $\cA$ }by
$$
\si(\cA) = \min \{\alpha\kle\omega_1:\,\cA\sder{\alpha}\keq \emptyset\}
$$
if there exists $\alpha\kle\omega_1$ with $\cA\sder{\alpha}\keq
\emptyset$, and $\si(\cA)\keq\infty$ otherwise.

 Note that if $I_S(\cA)\not=\infty$, it will always be a successor ordinal. Indeed,
if $\lambda$ is a limit ordinal and $I_S (\cA^{\alpha})>0$ for all  $\alpha<\lambda$, then, since
 $\emptyset \in \bigcap_{\alpha<\lambda}  \cA^{(\alpha)}=\cA^{(\lambda)}$ we get  $I_S(\cA)>{\lambda}$.

If $\cA$ is a tree on $S_X$ we call a subset $\cB\subset \cA$ {\em a subtree } if it is also a tree on $S_X$.
Let $Y$ be another Banach space and let  $\cA\subset S^{<\omega}_X$ and $\cB\subset  S^{<\omega}_Y$ be trees on $S_X$ and $S_Y$, respectively.
We say that $\cA$ {\em order isomorphically embeds into $\cB$} if there is a
injective map
$\Psi:\cA\to \cB$, with the property that  $\Psi(\bx)\prec \Psi(\bz)$ if and only if $\bx\prec\bz$. In that case $\Psi$ is called an {\em order isomorphism from
$\cA$ to $\cB$}.

In this paper we will only  consider the case that $S$ consists of the  normalized weakly null sequences and will therefore write
$\cA'$ and $\cA^{(\alpha)}$, for a tree
$\cA\subset S_X^{<\omega}$, instead of $\cA_S'$ and $\cA^{(\alpha)}_S$, respectively, and we  put $I_w(\cA)=I_S(\cA)$, which we call the
{\em weak index of $\cA$}.

The following Proposition describes a sufficient  condition for $I_w(\cA)\le I_w(\cB)$, if  $\cA$, $\cB$  are two trees on the sphere of two
Banach spaces $X$ and $Y$.

\begin{prop}\label{P:2.1}
Assume that $X$ and $Y$ are two Banach spaces, and   $\cA\subset S^{<\omega}_X$ and  $\cB\subset S^{<\omega}_Y$ are  trees on $S_X$ and $S_Y$, respectively, and assume that  there is an order isomorphism $\Psi$ from $\cA$ to $\cB$, with the  following  property:
\begin{align}\label{E:2.1.1} &\text{If $\bx\kin \cA$  and if $(x_k)\ksubset S_X$ is weakly null, with $(\bx, x_k)\kin\cA$, for $k\kin\N$, then there} \\
&\text{is a weakly null sequence $(y_k)\ksubset S_Y$,
so that $\Psi(\bx,x_k)=(\Psi(\bx),  y_k)$.}\notag
\end{align}

Then   $I_w(\cA)\le I_w(\cB)$.
\end{prop}

\begin{proof} We verify by transfinite induction that  for all ordinals $\alpha$
\begin{align}\label{E:2.1.2}&\Psi\big(\cA^{(\alpha)}\big)\subset \cB^{(\alpha)}.\qquad\qquad\qquad\qquad
\end{align}
If  $\alpha=0$ this is  just our assumption.If \eqref{E:2.1.2}  holds for  some ordinal $\alpha$ and if $\bx\in \cA^{(\alpha+1)}$,
then there is a weakly null sequence $(x_k)\subset S_X$ so that $(\bx,x_k)\in\cA^{(\alpha)}$, for all $k\kin\N$,  and, by
 assumption \eqref{E:2.1.1}, we can
choose a weakly null sequence $(y_k)\subset S_Y$
so that  $\Psi(\bx,x_k)= \big(\Psi(\bx),y_k)$. By the  induction hypothesis
$(\Psi(\bx),y_k)=\Psi(\bx,x_k)\in \cB^{(\alpha)}$ for all $k\in\N$. Now, since $(y_k)$ is weakly null, this implies  that $\Psi(\bx)\in \cB^{(\alpha+1)}$.

If $\lambda$ is a limit ordinal and \eqref{E:2.1.2}  holds for all   $\alpha<\lambda$, then
$$\Psi\big(\cA^{(\lambda)}\big)= \bigcap_{\alpha<\lambda}  \Psi\big(\cA^{(\alpha)}\big)\subset
\bigcap_{\alpha<\lambda}  \cB^{(\alpha)}=\cB^{(\lambda)}.$$
\end{proof}

The following characterization of  the Szlenk index was proven in \cite{AJO}.
\begin{thm}\label{T:2.2}\cite[Theorem 4.2]{AJO} If $X$ is a separable Banach space not containing $\ell_1$ then
 $$\sz(X)=\sup_{\rho>0} I_w(\cF_\rho),$$
 where for $\rho>0$, we let
 $$\cF_\rho=\cF^X_\rho=\Big\{(x_1,x_2,\ldots x_n)\in S_X^{<\omega}: \Big\|\sum_{i=1}^n a_i x_i\Big\|\ge\rho \sum_{i=1}^n a_i, \text{ for all $(a_i)_{i=1}^n\subset[0,\infty)$}\Big\}.$$
\end{thm}

It is important to note, and it will be used repeatedly in what follows, that $\cF_\rho$ is in fact an H-tree, and , thus, that all its derivatives are H-trees.

\begin{rem} In \cite[Definition 3.6]{AJO} the set $\cF_\rho$ was actually defined differently, namely
$$\tilde\cF_\rho=\left\{(x_1,x_2,\ldots x_n)\in S_X^{<\omega}: \begin{matrix} \big\|\sum_{i=1}^n a_i x_i\big\|\ge\rho \sum_{i=1}^n a_i ,\text{ for all $(a_i)_{i=1}^n\subset[0,\infty)$}\\
\text{ and } (x_1,x_2,\ldots x_n) \text{ is $\frac1\rho$-basic}\end{matrix}
 \right\}.$$
This was necessary in \cite{AJO} because in that paper the $S$-derivatives for  several  other  sets $S$ of sequences were considered.\end{rem}

In the case that one only considers derivatives with respect to the weakly null sequences the restriction to $\frac1\rho$-basic sequences is superfluous, as the next
proposition shows.

\begin{prop}\label{P:2.3} Let $\cA\subset S_X^{<\omega}$ be an H-tree and $c>1$. Then
$$I_w(\cA)=I_w\big(\cA\cap \{(x_1,x_2,\ldots x_n)\in S^{<\omega} : (x_1,x_2,,\ldots x_n)\text{ is $c$-basic}\}\big).$$
\end{prop}

\begin{proof} For $c>1 $ and a  finite dimensional subspace $F$ of  $X$  we put
$$\cA_{(F,c)}=\left\{(x_1,x_2,\ldots x_n)\in \cA: \begin{matrix}
                                                                         \big\|a_0 y_0 + \sum_{i=1}^ma_i  x_i\big\|\le c  \big\|a_0 y_0 + \sum_{i=1}^n a_i x_i\big\|_{\phantom{A}}\\
                                                                                                                                              \text{for all } y_0\in  F, \,(a_i)_{i=0}^n\subset \R,
                                                                            \text{ and }0\le m\le n\,\, \end{matrix} \right\}.$$
By transfinite induction we will show that for all $\alpha<\omega_1$, if $\cA^{(\alpha)}\not=\emptyset $, then
 $\cA^{(\alpha)}_{(F,c)}\not=\emptyset $, for all $c>1$ and all finite dimensional subspaces $F\subset X$. Then our claim follows simply by letting $F=\{0\}$.

 If $\cA=\cA^{(0)}\not=\emptyset$, then $\emptyset\in \cA$ and thus
 $\emptyset\in \cA_{(F,c)}$, for all $c>1$ and all finite dimensional subspaces $F\subset X$.

 Assume that our claim is true for  some ordinal $\alpha$ and assume that $\cA^{(\alpha+1)}\not=\emptyset $. Let $F\subset X$ be finite dimensional and
 $c>1$. Choose $c'=\sqrt{c}$. Since $\emptyset \in \cA^{(\alpha+1)}$, there exists a weakly null sequence $(y_j)\subset \cA^{(\alpha)}$
 and, thus, $\big(\cA(y_j)\big)^{(\alpha)}=\cA^{(\alpha)}(y_j)  \not=\emptyset$, for all $j\kin\N$.
 Put $F_j=\spa(F\cup\{y_j\})$, for  $j\kin\N$. From the
 induction hypothesis we deduce that
 $\big(\cA(y_j)\big)^{(\alpha)}_{(F_j,c')}\not=\emptyset$, for all $j\in\N$. But now we note that
 $(x_1,\ldots,x_n)\in \big(\cA(y_j)\big)^{(\alpha)}_{(F_j,c')}$, for some $j\in\N$, means that $(x_1,x_2,\ldots x_n)\in \big(\cA(y_j)\big)^{(\alpha)}$
 and $\big\|a_0 y + \sum_{i=1}^mx_i\big\|\le c'  \big\|a_0 y + \sum_{i=1}^n x_i\big\|$, for all  $y\in {F_j}, \,(a_i)_{i=0}^n\subset \R,
                                                                            \text{ and } m\le n $. The first condition means that $(y_j,x_1,\ldots x_n)\in \cA^{(\alpha)}  $. Since $(y_j)$ is weakly null the second condition implies for large enough $j_0\in\N$ and $j\ge j_0$ that
   \begin{align*}
   &\big\|b_0 y\big\|\le c' \big\|b_0 y+b_1 y_j\big\|\le c' c'                   \Big\|b_0 y+b_1 y_j+\sum_{i=1}^n a_i x_i\Big\|,
      \end{align*}
for all $y\kin F$ and $b_0,b_1,a_1,a_2,\ldots a_n\kin \R$.
  Thus
   $(y_j, x_1,\ldots x_n)\in \cA^{(\alpha)}_{(c,F)}$ for all $j\ge j_0$. We deduce that $\cA^{(\alpha+1)}_{(c,F)}\not=\emptyset$, which finishes the
   induction step   for successor ordinals. If $\lambda$ is a limit ordinal and
   $\cA^{(\lambda)}\not=\emptyset$ it follows that $\emptyset \in\cA^{(\alpha)}$,  for all $\alpha<\lambda$, and thus, by the induction hypothesis
    $\emptyset \in\cA^{(\alpha)}_{(F,c)}$ for any $c\kgr1$ and
   finite dimensional  subspace $F\ksubset X$, which implies that      $\emptyset \in\bigcap_{\alpha<\lambda} \cA^{(\alpha)}_{(F,c)}= \cA^{(\lambda)}_{(F,c)}$.
   This finishes the induction step, and the proof of our claim.  \end{proof}

The following Lemma compares   the weak index of trees which are in a certain sense close to each other.
\begin{lem}\label{L:2.4}
Let $X$ and $Y$ be subspaces of a Banach space $Z$,  with  $X^*$ and $Y^*$ being  separable, and  let $\vp>0$.
Assume that   $\dist(x,Y)<\vp\|x\|$, for each $x\in X$.

Then it follows for any H-tree  $\cA$ on $S_X$, with $I_w(\cA)<\infty$, that
$I_w(\cA)\le I_w(\cB)$, where
$$\cB=\big\{(y_1,y_2,\ldots y_n)\in S_Y^{<\omega}: \exists (x_1,x_2,\ldots x_n)\in \cA\quad\|x_j-y_j\|\le 4\vp,\text{ for }j=1,2\ldots n\big\}.$$

\end{lem}

\begin{proof}We first prove  the following

\noindent{\bf Claim 1.}  For every weakly null sequence $(x_j) \subset S_X$ there is a  subsequence $(x'_k)$ of  $(x_j)$ and a weakly null sequence $(y_k)$ in $S_Y$
so that $\|x'_k-y_k\|\le 4 \vp$.

For a Banach space $U$ we denote the weak topology on $U$ by $\sigma(U,U^*)$ and the weak$^*$ topology on $U^*$ by
$\sigma(U^*,U)$.
By assumption we can find  $\xt_j\in Y$, for every $j\kin\N$, with $\|\xt_j-x_j\|<\vp$. We choose an element
$$z^{**}\in \bigcap_{n=1}^\infty \overline{ \{\xt_j-x_j: j\ge n\big\}}^{\sigma(Z^{**},Z^*)}\subset \vp B_{Z^{**}}$$
(considering $Z$ as a subspace of $Z^{**}$ via the canonical map).
We let  $I=\N\times\cU$, where $ \cU$ is a neighborhood basis of $0$ in $\sigma(Z^{**}, Z^*)$, and consider the order  on $I$ defined  by
$(n,U)\le (n',U')$ if and only if $n\le n'$ and $U\supset U'$. We pick for every $\iota=(n,U)\in I $
 an element  $\xt_\iota-x_\iota\in    \{\xt_j-x_j: j\ge n\big\} \cap (z^{**}+U)   $ and note that
  $(\xt_\iota-x_\iota:\iota\in I)$ is a net which  $\sigma(Z^{**},Z^*)$-converges
 to $z^{**}$. Since $(x_j)$ is $\sigma(X,X^*)$-null, it follows that
$  \sigma(Z^{**},Z^*)-\lim_{\iota\in I} x_\iota=0$, and thus, since $Y^{**}$ is $\sigma(Z^{**}, Z^*)$-closed in $Z^{**}$,
$$z^{**}= \sigma(Z^{**},Z^*)-\lim_{\iota\in I} \xt_\iota-x_\iota=\sigma(Z^{**},Z^*)-\lim_{\iota\in I} \xt_\iota \in Y^{**}.$$
Since $Y^*$ is separable the $\sigma(Y^{**},Y^*)$-topology is metrizable on $B_{Y^{**}}$, and we can find by
Goldstine's Theorem  a sequence
$(u_n)\subset \vp B_Y$ which $\sigma( Y^{**},Y^*)$- converges to $z^{**}$. This implies that
$0\in \bigcap_{n\in\N} \overline{\{\xt_j-u_k: j,k\ge n\}}^{\sigma(Y^{**},Y^*)}$, and using again the separability of $Y^*$ we can find strictly increasing
sequences $m(k)$ and $n(k)$ such that $(\xt_{m(k)}-u_{n(k)})_{k\in\N}$ converges in $\sigma(Y,Y^*)$  to $0$.
We deduce now our claim by letting $x'_k =x_{m(k)}$, and $y_k= (\xt_{m(k)}-u_{n(k)})/\|\xt_{m(k)}-u_{n(k)}\|$, and  noting
that
$$\Big\|x_{m(k)}-\frac{\xt_{m(k)}-u_{n(k)}}{\|\xt_{m(k)}-u_{n(k)}\|}\Big\|\le \|x_{m(k)}-\xt_{m(k)}\|+\|u_{n(k)}\|+
\big|\|\xt_{m(k)}-u_{n(k)}\|-1\|\big|\le 4\vp.$$

Next we prove the following claim by transfinite induction for all ordinals $\alpha$, which  will yield, together with Proposition \ref{P:2.1}, the assertion of our lemma.

\noindent{\bf Claim 2.} For any  H-tree $\cA$  on $S_X$,  with $I_w(\cA)= \alpha+1$, there exist a subtree
$\tilde \cA $ of $\cA$, and a length preserving order isomorphism $\Psi: \tilde A\to \cB$,
so that
\begin{align}
\label{E:2.4.1} &I_w(\tilde \cA)=I_w(\cA)=\alpha+1,\text{ and } \\
\label{E:2.4.2} &\Psi\text{ satisfies condition \eqref{E:2.1.1} of Proposition \ref{P:2.1}}. \qquad\qquad\qquad\qquad\qquad\qquad
\end{align}
If  $\alpha=0$ and $I_w(\cA)=1$, we simply can take $\tilde A=\{\emptyset\}$ and put $\Psi(\emptyset)=\emptyset$.
Assume now that our claim is true for $\alpha$ and  that $\cA$ is an H-tree with $I_w(\cA)= \alpha+2$.
We deduce, that $\emptyset \in \cA^{(\alpha+1)}$ and that there is a weakly null sequence  $(x_k)_{k\in\N}\subset S_X$, so that $x_k=(\emptyset, x_k)\in \cA^{(\alpha)} $, which means that $I_w(\cA(x_k))\ge \alpha+1$, for $k\in\N$.

After passing to a subsequence of $(x_k)$ we can, using  Claim 1, assume that there is  a weakly null sequence $(y_k)\subset S_X$ so that
$\|x_k-y_k\|\le 4\vp$, for all $k\kin\N$. After passing to a cofinite subsequence
of  $(x_k)$ we can assume that $I_w(\cA(x_k))= \alpha+1$, for all $k\in\N$. Indeed, otherwise we could pass to a subsequence $(x'_k)$ of
$(x_k)$, so that  $I_w(\cA(x'_k))\ge  \alpha+2$, for all $k\in\N$, which would imply that $x'_k\in \cA^{(\alpha+1)}$ for all $k\in\N$, and thus
 $\emptyset \in  \cA^{(\alpha+2)}$, which would mean that $I_w(\cA)\ge  \alpha+3$, a contradiction.

Applying the inductive hypothesis we find for every $k\kin\N$ a subtree $\tilde\cA_k$ of $\cA(x_k)$,
 with $I_w(\cA_k)=I_w(\tilde\cA_k)=\alpha+1$,
 and a length preserving isomorphism $\Psi_k: \tilde \cA_k\to \cB$, which satisfies \eqref{E:2.4.2}.

We glue these trees, and isomorphisms  together by letting
 \begin{align*}
 &\tilde \cA=\big\{ (x_k,\bx^{(k)}): k\in\N\text{ and } \bx^{(k)}\in \tilde \cA_k\big\}\cup\{\emptyset\}\text{ and } \\
 &\Psi: \tilde\cA \to \cB,\quad \bx\mapsto \begin{cases} \emptyset &\text{if } \bx=\emptyset, \\
                                                                              \big(y_k,\Psi(\bx^{(k)})\big) &\text{if $\bx\keq(x_k,\bx^{(k)})$, for some $k\kin\N$ and $\bx^{(k)}\kin \tilde \cA_k$}.
                                                                               \end{cases}\end{align*}
It is now routine to verify that $\tilde \cA$, $\cB$ and $\Psi$ satisfy  conditions
\eqref{E:2.4.1} and \eqref{E:2.4.2}.

In the case that $\alpha$ is a limit ordinal and we assume that our claim holds for all $\alpha'<\alpha $ we proceed as follows.
Assume that $I_w(\cA)=\alpha+1$.
Let $(\alpha_n)$ be a sequence in $[0,\alpha )$ which increases to $\alpha $. For each $n\in\N$, we can pick   a weakly null sequence $( u_{(n,j)})_{j\in\N}\subset S_X$, so that
 $u_{(n,j)}\in \cA^{(\alpha_n)}$, for all $n,j\in\N$. Since  $X^*$ is separable, the weak topology on $B_X$ is metrizable, and we can
find a diagonal sequence  $(x_n)=(u_{(n,j_n)})$ which is also weakly null. It follows that   $I_w(\cA(x_n))\ge \alpha_n$, for all $n\in\N$.
 After passing to a subsequence of $ (x_n)$ we can assume, again using Claim 1, that there is a weakly null sequence $(y_n)\subset S_Y$,
 so that $\||x_n-y_n\|\le 4\vp$, for all $n\in\N$.
After passing to  a cofinite subsequence of $(x_n)$ we can assume that $I_w(\cA(x_n))< \alpha $, for all $n\in\N$.
 Indeed, otherwise there is an infinite subsequence $(x'_n)$ of $(x_n)$, so that $I_w(\cA(x'_n))\ge  \alpha +1$
 (recall that $I_w(\cdot)$ takes only values among the successor ordinals), and thus $x'_n\in\cA^{(\alpha )}$, for all $n\in\N$, which implies that
$I_w(\cA)\ge \alpha +2$,  a contradiction.

We apply the inductive hypothesis  for each $n\in\N$ to $\cA_n$ in order to obtain a subtree $\tilde\cA_n$ of $\cA(x_n)$,
 with $I_w(\tilde \cA_n)=  I_w(\cA(x_n))$,  a tree $\cB_n$ on $S_Y$, and
 an  order isomorphism from $\tilde \cA_n$ onto  $\cB_n$,  so that the conditions \eqref{E:2.4.1}, and  \eqref{E:2.4.2} are satisfied.
We now can define $\tilde\cA$, $\cB$ and $\Psi$ as before to verify our claim   in the case that $\alpha$ is a limit ordinal.
\end{proof}
\begin{cor}\label{C:2.5}
Let $X$ and $Z$  be  Banach spaces and   $Y$  be  a  subspace of $Z$.
 Assume that $Y^*$ and $X^*$  are separable and assume   for some $\rho\in (0,1)$ and $\vp\in (0,\rho/6)$
 there is an embedding $i:X\to Z$, with  $\|i\|\cdot\|i^{-1}\|\le 1+\vp$, so that
 $\dist(i(x),Y)\le \vp\|x\|$, for all $x\in X$.
 Then
 \begin{equation}\label{E:2.5.1}
 I_w\big( \cF_\rho^X\big)\le  I_w\big( \cF_{\rho-6\vp}^Y\big).
 \end{equation}

\end{cor}

\begin{proof}  If $(x_1,x_2,\ldots x_n)\in \cF^X_\rho$, and we let $z_j=i(x_j)/\|i(x_j)\|$ for $j=1,2\ldots n$, we deduce that
for $(a_j)_{i=1}^n\subset [0,\infty)$ that
\begin{align*}
\Big\|\sum_{j=1}^n a_j z_j \Big\| &\ge \Big\| \sum_{j=1}^n a_j i(x_j)\Big\| - \sum_{j=1}^n a_j \big| \|i(x)\|\!-\!1\big|
\ge\Big( \frac{\rho}{1\!+\!\vp} \!-\!\vp\Big) \sum_{j=1}^n a_j  \ge( \rho\!-\!2\vp) \sum_{j=1}^n a_j.
\end{align*}
It follows therefore that
$I_w(\cF^X_\rho)\le I_w(\cF^{i(X)}_{\rho-2\vp})$.  Replacing $X$ by $i(X)$,  and $\rho$ by $\rho-2\vp$, we   can assume that  $X$ is a subspace of $Z$
and need to show that  $I_w\big( \cF_\rho^X\big)\le  I_w\big( \cF_{\rho-4\vp}^Y\big)$.

We apply Lemma \ref{L:2.4} to $\cA=\cF_\rho^X$ (recall that $\cF_\rho^X$ is an H-tree) and note that  the tree $\cB$ on $S_Y$, as defined in Lemma \ref{L:2.4} is a subtree
of   $ \cF_{\rho-4\vp}^Y$. Indeed, if $(y_1,y_2,\ldots y_n)\in \cB$, and if $(x_1,x_2,\ldots x_n)\in\cF^X_\rho$, is such that
$\|x_j-y_j\|\le 4\vp$, then for all $(a_j)_{j=1}^n \subset [0,\infty)$,
$$\Big\| \sum_{j=1}^n a_j y_j \Big\|\ge  \Big\| \sum_{j=1}^n a_j x_j \Big\|- 4\vp\sum_{j=1}^n a_j\ge \rho \sum_{j=1}^n a_j - 4\vp\sum_{j=1}^n a_j=(\rho-4\vp) \sum_{j=1}^n a_j.$$
\end{proof}

The conditions described by our previous  Lemma \ref{L:2.4} and Corollary \ref{C:2.5}  are fulfilled
in the situation described by the next theorem, which is essentially due to
Zippin \cite[Theorem 1.2]{Zi}. Our formulation is
explicitly due to Benyamini \cite[page 27]{Be}.

\begin{thm}\label{T:2.6}
Let $X$ be a space with separable dual  and $0<\vp<\frac12$, and let $K$ be a $w^*$-closed and  totally disconnected
subset of $B_{X^*}$ which is $(1-\vp)$-norming $X$.

Then
there exist
$\beta<\omega^{\sz (X,\frac\vp8)+1}$
and   a subspace $Y$ of $ C(K)$, isometric to
$C([0,\beta])$, so that
$$
\dist(i(x),Y)\le2\vp\|x\|,\text{ for $x\in X$},
$$
where $i:X\to C(K)$ is the embedding defined by $i(x)(x^*)\keq x^*(x)$, for $x^*\kin K$ and $x\kin X$.
\end{thm}
\begin{rem} The proof of  \cite[Theorem 1.2]{Zi} shows that for any Banach space $X$ with a separable dual, and $\vp>0$,
 we can find a $w^*$-closed totally disconnected $(1-\vp)$-norming
subset of $B_{X^*}$. An explicit construction of such  a set
$K\subset B_{X^*}$ can also be found in \cite[Lemmas 1.3 and 1.4]{DFH}.

Let us also note that in \cite{Zi} and \cite{Be} another index $\eta(X,\vp)$ was used, replacing  in the statement
of Theorem \ref{T:2.6} our index $\sz(X,\vp)$. But since it was shown for  $\eta(\vp,X)$
in \cite[page 22]{AJO} (note that in \cite{AJO} $\eta(\cdot,\cdot)$ was called $\eta'(\cdot,\cdot)$, while  $\sz(\cdot,\cdot)$ was named $\eta(\cdot,\cdot)$)
 that  $\eta(X,\vp)\le \sz(X,\vp)\le \eta(X,\vp/2)$ for all $\vp>0$, our statement of  Theorem \ref{T:2.6} follows from the statement in \cite{Be}.
\end{rem}

The final key ingredient of our proof is the actual
computation of the dentability index of $C([0,\alpha])$,
$\alpha<\omega_1$, which was done in \cite[Proposition 12]{HLP}.
\begin{equation}\label{E:10}\text{Dz}(C([0,\omega^{\omega^\alpha}]))=
\sz(L_p(C([0,\omega^{\omega^\alpha}])))= \begin{cases}  \omega^{\alpha+2} &\text{ if $\alpha$ is finite,}\\
                                                            \omega^{\alpha+1}       &\text{ if $\alpha$ is an infinite ordinal.}\\
\end{cases}
\end{equation}

We are now ready for the proof of Theorem \ref{T:1.2}.

\begin{proof}[Proof of Theorem \ref{T:1.2}]
Lancien showed in \cite[Propositions 3.1 and 3.2]{La2} that $\sz(X)$ and $\dz(X)$  are
separably determined, provided they are countable, so
we may assume without loss of generality that $X$ is separable.

We will show that for any $\rho\in(0,1)$    and $1<p<\infty$ we
have
\begin{equation}\label{E:1.2.2}
I_w\big(\cF_\rho^{L_p(X)}\big)< \begin{cases} \omega^{\alpha} &\text{if  $\alpha$ is infinite,}\\
                     \omega^{\alpha+1} &\text{if $\alpha$ is finite.} \end{cases}
\end{equation}
Then our claim  follows from Theorem \ref{T:2.2}.

Let $\vp\in(0, \rho/12)$ and  apply Theorem  \ref{T:2.6},  which provides us with a $w^*$-closed, totally disconnected and $(1-\vp)$-norming $X$  subset $K$ of $B_{X^*}$, an ordinal  $\beta<\omega^{\sz(X,\vp/8)+1}$ and a subspace $Y\subset C(K)$, isometric to $C([0,\beta])$, so that   $\dist(i(x),Y)\le 2 \vp\|x\|$, for
all $x\in X$,  where $i: X\to C(K)$, is defined by $i(x)(x^*)=x^*(x)$, for $x^*\in K$ and $x\in X$.
Since $\beta<\omega^{\sz(X,\vp/8)}<\omega^{\omega^\alpha}$ equation
  \eqref{E:2.3}  yields that
 $\sz(Y)\le  \omega^{\alpha}$.
Indeed, if $\alpha=\gamma+1$ for some $\gamma<\omega_1$, then $\beta< \omega^{\omega^\gamma\cdot k}$ for some $k\in\N$ and we deduce
from   \eqref{E:2.3}  that
$\sz(Y)\le \sz \big(C([0, \omega^{\omega^\gamma\cdot k}\big])\big)= \sz \big(C([0, \omega^{\omega^\gamma }\big])\big)=
 \omega^{\gamma+1}=\omega^{\alpha} $. On the other hand, if
 $\alpha$ is a limit ordinal we deduce that $\beta<\omega^{\gamma}$, for some $\gamma<\alpha$, and we derive our claim the same way.

We define
 $$I: L_p(X)\to  L_p(C(K)), \qquad f\mapsto i\circ f,$$
and note that
$\|f\| (1-\vp) \le\| I(f) \|\le \|f\|$ for $f\in L_p(X)$. Observe  that  $L_p(Y)$ embeds naturally and isometrically into $L_p(C(K))$
and that
$\dist\big(I(f), L_p(Y)\big)\le 2\vp\|f\|$ for all $f\in L_p(X)$. The last observation follows easily for step functions and from the fact that $[0,1]$ with the Lebesgues measure is a probability space, and for general elements of $L_p(X)$ by approximation.

This means  that the spaces $\tilde X= L_p(X)$, $\tilde Y= L_p(Y)$ and $\tilde Z=L_p(C(K))$, and
 the isomorphic embedding  $I$ satisfy the conditions of Corollary \ref{C:2.5}  for $2\vp$ instead of $\vp$.
 We conclude therefore from   Corollary \ref{C:2.5} and \eqref{E:10}  that
 $$I_w\big( \cF^{L_p(X)}_\rho\big)\le I_w\big( \cF^{L_p(Y)}_{\rho-12\vp}\big)<\sz\big(L_p(Y)\big)\le
  \begin{cases}  \omega^{\alpha+1} &\text{ if $\alpha$ is finite,}\\
                                                            \omega^{\alpha}       &\text{ if $\alpha$ is an infinite ordinal.}\end{cases}
 $$ which proves our claim and finishes the proof of our theorem.
 \end{proof}

We remark that if $\alpha$ is finite and $\sz (X)=\omega^\alpha$ then
the precise value of $\text{Dz}(X)$ depends on the geometry of $X^*$. Indeed,  since
 $L_2(L_2(X))$ and $L_2(X)$ are isomorphic  for any Banach spaces $X$, it follows that
 $\sz(L_2(L_2(X))= L_2(X)$. But
  for $X=C[0,\omega^{\omega^\alpha}]$, where $\alpha$ is
  finite,   it was shown in  \cite[Theorem 2]{HLP} that $\text{Dz}(X)>\sz (X)$ and thus
  by Theorem \ref{T:1.2}
 $\text{Dz}(L_2(X))=\sz (L_2(X))= \omega^{\alpha+2}$.

\end{document}